\documentclass{amsart}

\usepackage{amsmath}
\usepackage{amssymb}
\usepackage{enumerate}
\usepackage{graphicx}
\usepackage{tikz}

\title[Divergent points of higher dimensional Collatz mappings]
 {A geometric approach to divergent points of higher dimensional Collatz mappings} 

\author[S. Kionke]{Steffen Kionke}

\address{Heinrich-Heine-Universit\"at \\ Mathematisches Institut \\ Universit\"atsstr. 1 \\ 40225 D\"usseldorf \\ Germany.}
\email{steffen.kionke@uni-duesseldorf.de}

\thanks{The research was supported by DFG grant KL 2162/1-1.}

\date{\today}
\subjclass[2010]{Primary 11B83; Secondary 11H06, 11B37}
\keywords{Collatz conjecture, 3x+1 problem}

\theoremstyle{plain}
\newtheorem{theorem}{Theorem}
\newtheorem*{theorem*}{Theorem}
\newtheorem{lemma}{Lemma}

\newtheorem{proposition}{Proposition}

\theoremstyle{definition}
\newtheorem{definition}{Definition}
\newtheorem*{conjecture}{Conjecture}
\newtheorem*{remark}{Remark}
\newtheorem{example}{Example}

\DeclareMathOperator{\vol}{vol}

\DeclareMathOperator{\Diver}{Div}
\DeclareMathOperator{\stp}{Stop}
\DeclareMathOperator{\GL}{GL}

\providecommand{\bbN}{\mathbb{N}}
\providecommand{\bbR}{\mathbb{R}}

\providecommand{\bbQ}{\mathbb{Q}}
\providecommand{\bbZ}{\mathbb{Z}}

\begin{document}
\begin{abstract}
We define generalized Collatz mappings on free abelian groups of finite rank and study their iteration trajectories.
Using geometric arguments we describe cones of points having a divergent trajectory and
we deduce lower bounds for the density of the set of divergent points.
As an application we give examples which show that the iteration of generalized Collatz 
mappings on rings of algebraic integers can behave quite differently from the conjectured behavior in $\bbZ$.  
\end{abstract}

\maketitle

\section{Introduction}

Consider the map $T\colon \bbZ \to \bbZ$ defined by $T(x) = \frac{x}{2}$ if $x$ is even and $T(x) = \frac{3x+1}{2}$ if $x$ is odd.
The famous $3x+1$ problem asks whether the iteration trajectory $T(n), T^2(n), T^3(n), \dots$ of every positive integer $n$ will eventually reach the value $T^k(n) = 1$.
This problem, known also as \emph{Collatz's problem}, is unsolved.
Many authors studied it and we refer to
the various survey articles \cite{ChamberlandSurvey, LagariasGeneralizations, LagariasOverview, MatthewsSurvey} for
an extensive introduction to the problem.

Matthews and Watts \cite{MatthewsWatts1984} (pursuing work of M\"oller \cite{Moller1978}) introduced \emph{generalized Collatz mappings} $T_{d,m,r}\colon \bbZ \to \bbZ$ defined as
\begin{equation*}
   T_{d,m,r}(x) = \frac{m_i x + r_i}{d} \quad \text{ if } x \equiv i \bmod d,
\end{equation*}
where $d> 1$ is an integer and $m = (m_0, \dots, m_{d-1})$ and $r = (r_0,\dots,r_{d-1})$ are
$d$-tuples of integers such that $m_i \neq 0$ and $m_i i + r_i \equiv 0 \bmod d$ for all $i$.
An integer $n$ is said to be \emph{eventually cyclic} for $T_{d,m,r}$ if the sequence $(T^i_{d,m,r}(n))_{i\in\bbN}$ is eventually periodic --
otherwise $n$ is called \emph{divergent}.
An integer $n$ is said to have \emph{finite stopping time} for $T_{d,m,r}$ if $|T^k_{d,m,r}(n)| < |n|$ for some $k$.
Clearly, if all but finitely many integers have finite stopping time, then every trajectory is eventually periodic.
It is known that, if the $m_i$ and $d$ are coprime and $|m_0 m_1 \cdots m_{d-1}| < d^d$, then the density of the set of integers with finite stopping time 
for $T_{d,m,r}$ is $1$ (see \cite{Heppner1979, Moller1978} in general and \cite{Everett1977, Terras1976} for the original problem).
This result and numerical experiments lead many authors to variations of the following conjecture. 
\begin{conjecture}[cf.~Conjecture 3.1 in \cite{MatthewsSurvey}]
  If $d$ and $m_i$ are coprime for all $i$ and $|m_0 m_1 \cdots m_{d-1}| < d^d$, then every integer is eventually cyclic for $T_{d,m,r}$.
\end{conjecture}

In this article we show that the direct generalization of this conjecture
to the ring of integers of any algebraic number field $K \neq \bbQ$ does not hold.
More precisely, we study generalized Collatz mappings on free abelian groups of rank 
greater than one.
Using a geometric argument we obtain a lower bound for the lower asymptotic density 
of the set of divergent points.
On the other hand we generalize the known results on the density of points with
finite stopping time to the higher dimensional situation.
This enables us to construct, for any given $\varepsilon > 0$, an example of
a generalized Collatz mapping on a lattice of rank $\geq 2$ such that
the set of elements with finite stopping time has density $1$, whereas
the set of divergent elements has lower asymptotic density at least $1-\varepsilon$.

\subsection{The main result}
Let $\Lambda$ be a free abelian group of finite rank $e$. 
We consider $\Lambda$ as a lattice in the real vector space $E = \Lambda \otimes_\bbZ \bbR$.
We equip $E$ with some norm $\Vert \cdot \Vert$, which will be used to define the asymptotic density of subsets of $\Lambda$.
Fix an integer $d > 1$. For every coset $\omega \in \Lambda/d\Lambda$
we choose an integer $m_\omega > 0$ and a vector $r_\omega \in \Lambda$ such that 
$m_\omega \omega + r_\omega \equiv 0 \bmod d\Lambda$. The map $T= T_{d,m,r} \colon \Lambda \to \Lambda$
defined by
\begin{equation*}
  T(x) = \frac{1}{d}(m_\omega x + r_\omega) \quad \text{ if } x \in \omega
\end{equation*}
will be called the \emph{generalized Collatz mapping} associated with $(d, m, r)$.
We study the iteration trajectories $(T^k(x))_{k\in \bbN}$ for elements $x \in \Lambda$. 
In particular, we are interested in the set $\Diver(T)$ of points which have a divergent $T$-trajectory.

We say that $T$ is of \emph{relatively prime type},
if for all $\omega \in \Lambda/d\Lambda$ the numbers $d$ and  $m_\omega$ are coprime (cf.~\cite{MatthewsSurvey}).
By $S_r$ we denote the set $S_r = \{ \: r_\omega \:|\: \omega \in \Lambda/d\Lambda \:\}$ of all shift vectors.
We say that $S_r$ is \emph{acute} if all non-zero 
vectors in $S_r$ lie in some open half-space. Let $B_r^+$ denote the closed convex cone spanned by $S_r$ and put $B_r^- = - B_r^+$.

Assume that $T$ is of relatively prime type and $S_r$ is acute.
We construct finitely 
many hyperplanes $H_1,\dots,H_k$ in $E$ (called \emph{separating} hyperplanes) and we study the decomposition of
 $E \setminus \bigcup_{j=1}^k H_j$ into its connected components
\begin{equation*}
 E \setminus \bigcup_{j=1}^k H_j = \bigcup_{i=1}^N C_i
\end{equation*}
where the components $C_1, \dots, C_N$ are open convex cones.
We say that the cone $C_i$ is \emph{wild} if $C_i \cap B^+_r \neq \emptyset$ and we define the
\emph{wild cone} $W_T$ of $T$ as the union of $B_r^-$ with the closures of the wild cones $C_j$.
The complement $D_T$ of the wild cone in $E$ will be called the \emph{tame cone} of $T$.

\begin{theorem*}
Assume that $T$ is of relatively prime type and $S_r$ is acute. Every point $x \in \Lambda \cap D_T$ has a divergent $T$-trajectory.
The lower asymptotic density of $\Diver(T)$ satisfies
\begin{equation*}
   \underline{\delta}(\Diver(T)) \geq \frac{\vol(D_T \cap \beta(1))}{ \vol(\beta(1))}
\end{equation*}
where $\beta(1)$ denotes the unit ball in $E$ with respect to the chosen norm.
\end{theorem*}

The result is a combination of Theorems \ref{thm:divergent} and \ref{thm:divDensity} proved
in Section \ref{sec:divergentCone}.
We generalize the known density results for the set of elements with finite stopping time in Section \ref{sec:stoppingTime}.
In Section \ref{sec:examples} we construct examples of generalized Collatz 
mappings where almost all elements have finite stopping time, but the density of divergent points is arbitrarily close to $1$.

\section{The wild and tame cone of a generalized Collatz mapping}\label{sec:divergentCone}
Let $\Lambda$ be a free abelian group of finite rank $e$ and
let $d > 1$ be a fixed integer. 
We write $E = \Lambda \otimes_\bbZ \bbR$ and $E_\bbQ$ denotes the $\bbQ$-span of $\Lambda$ in $E$.
We will illustrate our results and definitions with the following example of a generalized Collatz mapping.

\begin{example}\label{theExample}
  Consider the natural extension of the original Collatz
  mapping to the ring $\bbZ[\sqrt{2}]$ (considered already in \cite[Sect.~10]{MatthewsSurvey}).
  This is the map $C\colon \bbZ[\sqrt{2}] \to \bbZ[\sqrt{2}]$ defined by 
  $C(z) = z/\sqrt{2}$ if $z$ is divisible by $\sqrt{2}$ and $C(z) = \frac{3z+1}{\sqrt{2}}$ otherwise.
  Let $\Lambda = \bbZ[\sqrt{2}]$ and consider the mapping $T = C^2 \colon \Lambda \to \Lambda$.
  With respect to the $\bbZ$-basis $(1, \sqrt{2})$ of $\Lambda$ the mapping $T$ takes the form
  \begin{equation*}
      T(x,y) = \begin{cases}
                  \frac{1}{2}(x,y)     \quad& \text{ if $x,y$ even}  \\
                  \frac{1}{2}(3x+1,3y) \quad& \text{ if $x$ odd, $y$ even }\\
                  \frac{1}{2}(3x,3y+1) \quad& \text{ if $x$ even, $y$ odd }\\
                  \frac{1}{2}(9x+3,9y+1) \quad& \text{ if $x,y$ odd } \\
               \end{cases}.
  \end{equation*}
  Thus $T$ is a generalized Collatz mapping of relatively prime type.
\end{example}

\subsection{Semipermeable hyperplanes and directed points}

\begin{definition}
 Let $S \subseteq E$ be a set of vectors.
 A non-trivial linear form $\Phi\colon E \to \bbR$ is called \emph{positive} for $S$, if $\Phi(w) \geq 0$ for all $w \in S$. 
 We say that $\Phi$ is \emph{strictly positive} for $S$ if $\Phi(w) > 0$ for all non-zero $w \in S$.
 A set $S$ which admits a strictly positive form is called \emph{acute}.
\end{definition}
 The kernel $H$ of a strictly positive form $\Phi$ for the set $S_{r} = \{ \: r_\omega \:|\: \omega \in \Lambda/d\Lambda \:\}$ will be called 
 a \emph{semipermeable} hyperplane for $T_{d,m,r}$.
 This name is inspired by the following simple observation:
 If $\Phi$ is a positive form for $S_{r}$, then
 no trajectory of $T_{d,m,r}$ can pass from the positive half-space of $\Phi$ to the negative half-space.
 In particular, we obtain the following result.
 
 \begin{lemma}
   Assume that $T = T_{d,m,r}$ is of relatively prime type.
   Every point $x\neq 0$ that lies on a semipermeable hyperplane does not have a cyclic $T$-trajectory.
 \end{lemma}
 \begin{proof}
   Suppose $x \neq 0$ lies on the semipermeable hyperplane $H$ and
   let $\Phi$ be a strictly positive form for $S_r$ with kernel $H$.
   Define $x_0 = x$ and consider the $T$-trajectory $x_{i+1} = T(x_i)$ of $x$.
   Put $\omega_i = x_i + d\Lambda$.
   We claim that $\Phi(x_i) > 0 $ for every sufficiently large integer $i$; henceforth $x_i \neq x$ for such $i$.
    
   Clearly, the equation 
   \begin{equation}\label{eq:stepTrajectory}
    \Phi(x_{i+1}) = \frac{m_{\omega_i}}{d} \Phi(x_i) + \frac{1}{d} \Phi(r_{\omega_i})
   \end{equation}
   and the positivity of $\Phi$
   yield $\Phi(x_{i+1}) > 0$ whenever $\Phi(x_i)>0$. It remains to show that $\Phi(x_j) > 0 $ for some $j$.
   
   Suppose $\Phi(x_i) = 0$ for all $i$. Equation \eqref{eq:stepTrajectory} implies that $\Phi(r_{\omega_i}) = 0$ for all $i$
   and, as $\Phi$ is strictly positive, we conclude that $r_{\omega_i} = 0$ for each index $i$.
   The mapping $T$ is of relatively prime type, so $r_{\omega_i} = 0$ implies $\omega_i = 0$.
   This finishes the proof as it
   contradicts our assumption $x \neq 0$.
 \end{proof}

 \begin{definition}
   A point $x \in \Lambda \setminus \{0\}$ is called \emph{directed} for $T_{d,m,r}$ if it lies on a semipermeable hyperplane.
 \end{definition}
 
  We give a geometric description of the set of directed points.
  Note that, if directed points exist, then $S_r$ is acute.
  Let $B_r^+$ be the closed convex cone generated by $S_r$ in $E$, this is, 
  the intersection of all closed linear half-spaces containing the set $S_r$.
  We write $B_r^-$ for $-B_r^+$.
  
  \begin{lemma}\label{lem:directed}
    Suppose that $S_r$ is acute.
    A point $x \in \Lambda$ is directed for $T$ if and only if $x \notin B_r^+ \cup B_r^-$.
  \end{lemma}
  \begin{proof}
   ``$\Rightarrow$'': Let $x \in \Lambda$ be directed and pick a strictly positive form $\Phi$ for $S_r$ with $\Phi(x) = 0$.
   Since $x \neq 0$, there is a linear form $\alpha\colon E \to \bbR$ with $\alpha(x) > 0$. For all sufficiently small $\varepsilon > 0$
   the linear form $\Phi - \varepsilon \alpha$ is strictly positive for $S_r$ and its positive half-space
   contains $S_r$ but not $x$. We conclude $x \notin B_r^+$. Similarly, $x \notin B_r^-$ using the form $-\Phi - \varepsilon \alpha$.
   
   ``$\Leftarrow$'': Assume that $x \notin B_r^+ \cup B_r^-$. We can find linear forms $\Phi_1, \Phi_2 \colon E \to \bbR$
   such that the closed positive half-space of $\Phi_1$ contains $S_r$ but not $x$, and the closed positive half-space of $\Phi_2$ contains $-S_r$ but not $x$.
   Since $S_r$ is acute, we can assume that $\Phi_1$ and $-\Phi_2$ are strictly positive for $S_r$.
   Finally, we choose positive real numbers $u_1$ and $u_2$ such that $\Phi = u_1 \Phi_1 - u_2 \Phi_2$ is strictly positive for $S_r$ and $\Phi(x) = 0$.
  \end{proof}
  
  \begin{example}\label{theExample2}
    Consider the mapping $T$ defined in Example \ref{theExample}.
    The set of shift vectors is $S_r = \{(0,0), (1,0), (0,1), (3,1)\}$.
    It is acute; indeed, the form $\Phi(x,y) = x+y$ is strictly positive for $S_r$.
    The cone $B_r^+$ generated by $S_r$ is the first quadrant, the opposite cone $B_r^-$ is
    the third quadrant.
    All points in $\bbZ^2$ which lie in the second and fourth quadrant are hence directed and we conclude that
    there are no cycles in these quadrants.
    
    In an unpublished preprint of Kucinski \cite{Kucinski2000} a similar observation has been made for
    an extension of the original Collatz map to the ring $\bbZ[i]$ of Gaussian integers
    (the extension was already considered by Joseph in \cite{Joseph1998}).
  \end{example}

  We say that a subset $U \subset \Lambda$ is $T$-invariant if
  $T(U) \subseteq U$. The following result summarizes the discussion of this section.
  \begin{proposition}\label{prop:InvariantDirected}
    Assume that $T$ is of relatively prime type.
    If $U \subseteq \Lambda$ is a $T$-invariant subset of directed points, then every $T$-trajectory in $U$ is divergent.
  \end{proposition}
  \begin{proof}
    Every point $x \in U$ is directed, hence its trajectory is not cyclic. Since $U$ is $T$-invariant, no trajectory leaves $U$ and hence
    cannot be eventually cyclic.
  \end{proof}

\subsection{Separating hyperplanes}\label{sec:separatingHyperplanes}

In this section we discuss a geometric decomposition of
$\Lambda$ into $T$-invariant subsets using \emph{separating} hyperplanes.

\begin{definition}
  A non-trivial linear form $\Phi\colon E \to \bbR$ is called \emph{separating} for $T_{d,m,r}$ if for all $\omega \in \Lambda/d\Lambda$ and all
  $x \in \omega$ with $\Phi(x) \neq 0$ the inequality  $|m_\omega \Phi(x)| > |\Phi(r_\omega)|$ holds.
  A hyperplane $H \subset E$ is called separating if there is a separating linear form $\Phi$ with $\ker(\Phi) = H$.
\end{definition}
Whether a linear form $\Phi$ is separating only depends on the hyperplane $H = \ker(\Phi)$. 
Observe that no $T$-trajectory can pass from either side of a separating hyperplane to the other; 
indeed, for $\Phi(x)\neq 0$ the number $\Phi(T(x))$ has the same sign as $\Phi(x)$.
A hyperplane $H \subseteq E$ is said to be defined over $\bbQ$, if there is a $\bbQ$-linear form $\Phi\colon E_\bbQ \to \bbQ$
such that $\ker(\Phi)$ spans $H$ over $\bbR$.

\begin{proposition}\label{prop:separating}
 Assume that the shift vectors $S_r$ span $E$ as real vector space.
 Then there are only finitely many separating hyperplanes and every such hyperplane is defined over $\bbQ$.
\end{proposition}
\begin{proof}
  Let $H$ be a hyperplane which is not defined over $\bbQ$. We show that $H$ is not separating. 
  Take a linear form $\Phi\colon E \to \bbR$ with $\ker(\Phi) = H$. As $H$ is not defined over $\bbQ$, we have
  $\dim_\bbQ(\Phi(E_\bbQ)) \geq 2$ -- thus we can find a $\bbZ$-basis $u_1,\dots, u_e$ of $\Lambda$ such that 
  $\lambda_1 = \Phi(u_1)$ and $\lambda_2 = \Phi(u_2)$ are $\bbQ$-linearly independent real numbers.
  By assumption the vectors $(r_\omega)_{\omega \in \Lambda/d\Lambda}$ span $E$ over $\bbR$ and therefore we find some $\omega \in \Lambda/d\Lambda$ such that
  $\Phi(r_\omega) \neq 0$.
  Take any representative $x \in \omega$. The set of real numbers $\Phi(x) + \bbZ d\lambda_1 + \bbZ d \lambda_2$
  is dense in $\bbR$ and hence contains arbitrarily small positive elements. In particular, we find a vector $y\in \omega$
  of the form $y = x + dn_1 u_1 + dn_2 u_2$ with $|\Phi(y)| < \frac{|\Phi(r_\omega)|}{m_\omega}$.
  We conclude that $H$ is not separating.
  
  Finally, we show that the number of separating hyperplanes is finite.
  Let $H$ be a separating hyperplane. It is defined over $\bbQ$ and 
  thus we can find a defining $\bbQ$-linear form $\Phi\colon E_\bbQ \to \bbQ$ such that $\Phi(\Lambda) = \bbZ$.
  Hence for every coset $\omega \in \Lambda/d\Lambda$, the minimal non-zero value of $|\Phi(x)|$ for $x \in \omega$ is at most $d$.
  Since $\Phi$ is separating, we find $|\Phi(r_\omega)| < d m_\omega$ for all $\omega \in \Lambda/d\Lambda$.
  The form $\Phi$ is uniquely determined by its values on the spanning vectors $r_\omega$ and hence there is only a finite number of such forms $\Phi$.
  \end{proof}

\begin{example}\label{theExample3}
  The given proof can be easily transformed into a method to find all separating hyperplanes.
  We explain this by means of the mapping $T$ defined in Example \ref{theExample}.
  
  Let $\Phi\colon \bbR^2 \to \bbR$ be a separating linear form for $T$. By Proposition \ref{prop:separating} we can multiply $\Phi$ with a suitable scalar, such that
  $\Phi$ is of the form $\Phi(x,y) = ax + by$ with $a,b \in \bbZ$, $a \geq 0$ and $\gcd(a,b) = 1$. 
  We find conditions on the integers $a$ and $b$ so that the form $\Phi$ is separating.
  
  Let $u,v$ be integers. Since $\Phi( (u,v) + 2 \bbZ^2 ) = ua + vb + 2\bbZ$,
  the minimal non-zero absolute value of $\Phi$ on the coset $(u,v) + 2\bbZ^2$ equals $1$ if $ua+vb$ is odd and equals $2$ if $ua+vb$ is even.
  We distinguish three cases according to the parity of $a$ and $b$.
  
 \begin{figure}[!htb]
  \begin{tikzpicture}[x=1cm, y=1cm, semitransparent]
\draw[step=0.5cm, line width=0.2mm, gray] (0,0) grid (6,6);
\draw[ thick, -latex, opacity=0.3] (0,3) -- (6,3);
\draw[ thick, -latex, opacity=0.3] (3,0) -- (3,6);
\draw[opacity=0.8] (4,3)  node [above] {\tiny{$2$}};
\draw[opacity=0.8] (5,3)  node [above] {\tiny{$4$}};
\draw[opacity=0.8] (2,3)  node [above] {\tiny{$-2$}};
\draw[opacity=0.8] (1,3)  node [above] {\tiny{$-4$}};
\draw[opacity=0.8] (3,4)  node [right] {\tiny{$2$}};
\draw[opacity=0.8] (3,5)  node [right] {\tiny{$4$}};
\draw[opacity=0.8] (3,2)  node [right] {\tiny{$-2$}};
\draw[opacity=0.8] (3,1)  node [right] {\tiny{$-4$}};
\draw[thick, blue, opacity=0.7] (0,3) -- (6,3);
\draw[thick, blue, opacity=0.7] (3,0) -- (3,6);
\draw[thick, blue, opacity=0.7] (0,0) -- (6,6);
\draw[thick, blue, opacity=0.7] (0,6) -- (6,0);
\draw[thick, blue, opacity=0.7] (6,4.5) -- (0,1.5);
\draw[thick, blue, opacity=0.7] (6,1.5) -- (0,4.5);
\draw[thick, blue, opacity=0.7] (6,3.75) -- (0,2.25);
\draw[thick, blue, opacity=0.7] (6,2.25) -- (0,3.75);
\draw[thick, blue, opacity=0.7] (4.5, 6) -- (1.5,0);
\draw[thick, blue, opacity=0.7] (4.5, 0) -- (1.5,6);
\node[draw=black, circle,inner sep=1pt,fill, fill opacity=0.9] at (3,3) {};
\filldraw[fill=red, fill opacity=0.1, opacity=0.2] (3,3)
        rectangle (6,6);
\filldraw[fill=red, fill opacity=0.1, opacity=0.2] (3,3)
        rectangle (0,0);
\end{tikzpicture}
\caption{The 10 separating hyperplanes of $T$ (blue) and the wild cone (red).}\label{fig:sepExample}
\end{figure}
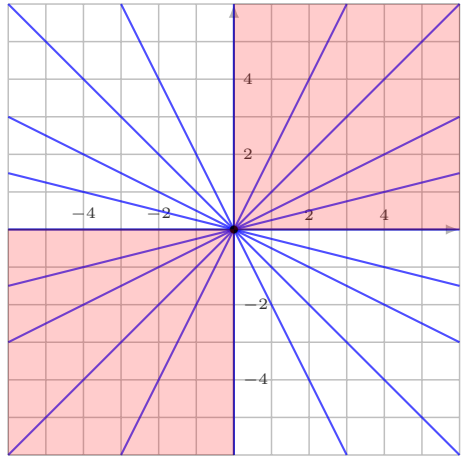

  \noindent\textbf{Case 1:} $a$ and $b$ are odd.\\
  \noindent In this case the minimal non-zero absolute value of $\Phi$ on $(1,0) + 2\bbZ^2$ and $(0,1)+2\bbZ^2$ is $1$ and
   the minimal non-zero absolute value of $\Phi$ on  $(3,1) + 2\bbZ^2$ is $2$.
   The form $\Phi$ is separating, exactly if the inequalities $0 \leq a < 3$, $|b| < 3$ and $|3a+b|<18$ are satisfied.
   We find two separating hyperplanes by taking $a=1$ and $b=\pm 1$.
   
   \noindent\textbf{Case 2:} $a$ is odd, $b$ is even.\\
   \noindent In this case the minimal non-zero absolute value of $\Phi$ on the $2\bbZ^2$-cosets of $(1,0)$ and $(3,1)$ is $1$. However, the minimal
   value attained on the coset of $(0,1)$ equals $2$.
   We obtain the inequalities $0 \leq a < 3$, $|b|<6$ and $|3a+b|<9$. We find five separating hyperplanes with $a=1$ and $b=0,\pm2, \pm4$.
   
  \noindent\textbf{Case 3:} $a$ is even, $b$ is odd.\\
  \noindent As above we find  $a = 0, 2$ and $b=\pm 1$. Both forms with $a=0$ define the same hyperplane and so we obtain three new separating hyperplanes.
  
  Observe that it is not possible that $a$ and $b$ are even, since
  $\gcd(a,b)=1$. In total there are 10 separating hyperplanes (see Figure \ref{fig:sepExample}).

\end{example}
Regarding the premise of Proposition \ref{prop:separating}, the following observation is useful.
\begin{lemma}
 If $T_{d,m,r}$ is of relatively prime type, then $S_r$ spans $E$ as $\bbR$-vector space.
\end{lemma}
\begin{proof}
 Consider the $\bbZ$-lattice $\Gamma$ generated by the set $S_r$. The lattice $\Gamma$ is clearly a sublattice of $\Lambda$.
 Since $T_{d,m,r}$ is relatively prime, we have $\omega \equiv - m_\omega^{-1} r_\omega \bmod d$.
 Thus the map $\Gamma \to \Lambda/d\Lambda$ is surjective. We deduce that $\Gamma$ is a sublattice of $\Lambda$ of full rank $e$; in particular, 
 it spans $E$ as vector space over $\bbR$.
\end{proof}

Finally, we are able to construct various $T$-invariant subsets.
\begin{lemma}\label{lem:sepCone}
  Let $\Phi_1, \dots, \Phi_\ell$ be linear forms on $E$ and assume that each $\Phi_i$ is positive or separating for $T$.
  Then the set
  \begin{equation*}
      C = \{\:x \in \Lambda\:|\: \Phi_i(x) > 0 \text{ for all } i=1,\dots,\ell \:\}
  \end{equation*}
  is $T$-invariant.
\end{lemma}
\begin{proof}
 Let $x \in C$ and put $\omega = x+ d\Lambda$.
 The formula $\Phi_i(T(x)) = \frac{m_\omega}{d} \Phi_i(x) + \frac{1}{d} \Phi_i(r_\omega)$ 
 yields $\Phi_i(T(x))>0$ whenever $\Phi_i(x) > 0$ using either that $\Phi$ is positive or that $\Phi$ is separating.
\end{proof}

Assume that $T = T_{d,m,r}$ is of relatively prime type and that $S_r$ is acute. 
Proposition \ref{prop:separating} implies that there are finitely
many separating hyperplanes $H_1,\dots,H_k$.
We decompose the complement $E \setminus \bigcup_{j=1}^k H_j$ into its connected components
\begin{equation}\label{eq:decomposition}
 E \setminus \bigcup_{j=1}^k H_j = \bigcup_{i=1}^N C_i
\end{equation}
where the components $C_1, \dots, C_N$ are open convex cones.
As before $B_r^+$ and $B_r^-$ denote the closed convex cones generated by $S_r$ and $-S_r$ respectively.
We say that the cone $C_i$ is \emph{wild} if $C_i \cap B^+_r \neq \emptyset$ and we define the
\emph{wild cone} $W_T$ of $T$ as
\begin{equation*}
   W_T = B_r^- \cup \bigcup_{\substack{i=1 \\ C_i \text{ wild }}}^N \overline{C_i}.
\end{equation*}
The complement $D_T$ of the wild cone $W_T$ in $E$ will be called the \emph{tame cone}.
\begin{theorem}\label{thm:divergent}
  Assume that $T$ is of relatively prime type and that $S_r$ is acute.
  Every element $x \in \Lambda \cap D_T$ has a divergent $T$-trajectory
\end{theorem}
\begin{proof}
  By construction, $D_T$ is a subset of $E \setminus B_r^+ \cup B_r^-$, hence by Lemma \ref{lem:directed} the set $D_T \cap \Lambda$ consists of directed points.
  According to Proposition \ref{prop:InvariantDirected} it suffices to show that $D_T \cap \Lambda$ is $T$-invariant.
  More precisely, we show that every element $x \in D_T\cap\Lambda$ lies in a $T$-invariant cone.
  
  Let $x \in D_T \cap \Lambda$. 
  Since $x$ is not in $B_r^-$, we can find a linear form $\Psi \colon E \to \bbR$ with $\Psi(x) > 0$ and $\Psi(-r_\omega) \leq  0$ for all $\omega \in \Lambda/d\Lambda$.
  In particular, $\Psi$ is a positive form for~$S_r$.
  Moreover, as $x$ does not lie in the closure of a wild cone, we can find
  separating linear forms $\Phi_1,\dots, \Phi_\ell$ such that
  the associated open positive cone contains $x$ but does not intersect $B_r^+$.
  By construction $x$ lies in the cone 
  \begin{equation*}
  C = \{\:y \in E \:|\: \Psi(y) > 0 \text{ and } \Phi_i(x) > 0 \text{ for all } i=1,\dots,\ell \:\} \subseteq D_T.
  \end{equation*}
  Lemma \ref{lem:sepCone} shows that $C \cap \Lambda$ is $T$-invariant, which finishes the proof.
\end{proof}

\subsection{Asymptotic density}
In order to understand \emph{how many} points of $\Lambda$ 
have a divergent $T$-trajectory, we consider the asymptotic density.
We fix some norm $\Vert \cdot \Vert$ on $E$.
The closed ball of radius $\rho$ around $0 \in E$ will be denoted by $\beta(\rho)$. 

For a subset $Z \subseteq \Lambda$ we define the asymptotic lower and upper density by
\begin{align*}
   \underline{\delta}(Z) = \liminf_{n \to \infty} \frac{|Z \cap \beta(n)|}{|\Lambda \cap \beta(n)|},\\
   \overline{\delta}(Z) = \limsup_{n \to \infty} \frac{|Z \cap \beta(n)|}{|\Lambda \cap \beta(n)|}.
\end{align*}
If the upper and lower density of $Z$ agree, then $\delta(Z) = \underline{\delta}(Z) = \overline{\delta}(Z)$
is called the asymptotic density of $Z$.

Fix some Lebesgue measure on $E$, which we denote by $\vol$.
The measure is unique up to positive scalars and our results are independent of this choice.
We briefly collect the following lemma about the density of lattice points in convex cones.
\begin{lemma}\label{lem:ConeDensity}
 Let $C \subseteq E$ be a closed convex cone and let $C^{0}$ denote its interior.
 The asymptotic densities of $\Lambda \cap C$ and $\Lambda \cap C^{0}$ exist and
 \begin{equation*}
    \delta(\Lambda \cap C) = \delta(\Lambda \cap C^{0}) = \frac{\vol\bigl(C\cap \beta(1)\bigr)}{\vol\bigl(\beta(1)\bigr)}.
 \end{equation*}
\end{lemma}
\begin{proof}
 The set $C_1 = C \cap \beta(1)$ is a convex body and is therefore Jordan measurable (see Theorem 7.4 in \cite{Gruber2007}).
 Let $F$ be a fundamental cell for $\Lambda$ in $E$ and let $e = \dim_\bbR E$.
 It follows from Section 7.2 in \cite{Gruber2007}, that for every integer $n \geq 1$
 \begin{align*}
      |\Lambda \cap C \cap \beta(n)| = n^e\frac{\vol(C_1)}{\vol(F)} + o(n^e), \\
      |\Lambda \cap \beta(n)| = n^e\frac{\vol(\beta(1))}{\vol(F)} + o(n^e).
 \end{align*}
 We deduce $\delta(\Lambda \cap C) = \frac{\vol(C_1)}{\vol(\beta(1))}$ from these equalities.
 
 Now we consider $\Lambda \cap C^{0}$.
 The interior of a Jordan measurable set is Jordan measurable with same measure. 
 Consequently, the set $C^{*}_1 = C^{0} \cap \beta(1)$ is Jordan measurable with $\vol(C^{*}_1) = \vol(C_1)$,
 as it is sandwiched between to $C_1$ and its interior. Note that, the Jordan measure is the Lebesgue measure on Jordan measurable sets.
 The same argument as above yields $\delta(\Lambda \cap C^{0}) = \frac{\vol(C^{*}_1)}{\vol(\beta(1))}$.
\end{proof}

\begin{theorem}\label{thm:divDensity}
  Assume that $T$ is of relatively prime type and that $S_r$ is acute.
  The set $\Diver(T)$ of points in $\Lambda$ which have a divergent $T$-trajectory satisfies
  \begin{equation*}
      \underline{\delta}(\Diver(T)) \geq \frac{\vol(D_T\cap\beta(1))}{\vol(\beta(1))},
  \end{equation*}
  where $D_T$ denotes the tame cone of $T$ defined in Section \ref{sec:separatingHyperplanes}.
\end{theorem}
\begin{proof}
 By Theorem \ref{thm:divergent} we have $D_T \cap \Lambda \subseteq \Diver(T)$, so it suffices to compute the
 density of $D_T \cap \Lambda$.
 Now we apply Lemma \ref{lem:ConeDensity} and use the fact that $D_T$ is defined in terms of convex cones to deduce
 \begin{align*}
   \delta(D_T\cap \Lambda) &= \sum_{\substack{i=1 \\ C_i \text{ not wild }}}^N \delta(C_i\cap \Lambda) - \delta(B^-_r \cap C_i \cap \Lambda)\\
                         &= \sum_{\substack{i=1 \\ C_i \text{ not wild }}}^N \frac{\vol\bigl( (C_i \setminus B_r^-)\cap \beta(1)\bigr)}{\vol\bigl(\beta(1)\bigr)}
                         = \frac{\vol(D_T\cap\beta(1))}{\vol(\beta(1))}.
\end{align*}
\end{proof}

\begin{example}\label{theExample4}
 We return to the Collatz mapping $T$ discussed in the Examples \ref{theExample}, \ref{theExample2} and \ref{theExample3}.
 The tame cone $D_T$ consists of the open second and fourth quadrants, whereas the wild cone consists of the closed first and third quadrant (see Figure \ref{fig:sepExample}).
 In particular, we obtain $\underline{\delta}(\Diver(T)) \geq \frac{1}{2}$.
\end{example}

\section{Points with finite stopping time}\label{sec:stoppingTime}
Let $\Lambda\subseteq E$ be a lattice in a normed vector space of dimension $\dim_\bbR E = e$.
A point $x \in \Lambda$ is said to have \emph{finite stopping time} for $T$ if
there is some integer $k > 1$ such that $\left\Vert T^k(x)\right\Vert < \left\Vert x\right\Vert$.
We denote the set of points with finite stopping time in $\Lambda$ by $\stp(T)$.

For the original $3x+1$-mapping Terras \cite{Terras1976} and Everett \cite{Everett1977} showed that $\stp(T)$ has asymptotic density
$\delta(\stp(T)) = 1$. Generalizing the work of M\"oller \cite{Moller1978} and
Heppner \cite{Heppner1979} to the higher dimensional setting one obtains the following theorem.
\begin{theorem}\label{thm:stoppingTime}
 Let $T = T_{d,m,r}$ be a generalized Collatz mapping of relatively prime type on a lattice $\Lambda \subseteq E$ with $\dim_\bbR E = e$ and assume that
 \begin{equation*}
     \prod_{\omega \in \Lambda/d\Lambda} m_\omega < d^{d^e}.
 \end{equation*}
 Then the set $\stp(T)$ has asymptotic density $\delta(\stp(T)) = 1$.
\end{theorem}
 The proof is a straightforward generalization of the one dimensional case.
 We follow along the lines of Everett \cite{Everett1977} and Heppner \cite{Heppner1979} to provide a short sketch.
 
 \begin{lemma}\label{lem:stopping}
   Suppose $T$ is of relatively prime type.
   Let $x_i = T^i(x)$ denote the $T$-trajectory of $x \in \Lambda$ and let $\omega(x,i) = x_i + d\Lambda$ denote the residual sequence. 
   \begin{enumerate}[(i)]
    \item\label{eq:formula} $x_k = d^{-k}\Bigl( (\prod_{i = 0}^{k-1} m_{\omega(x,i)})x + \sum_{j=0}^{k-1} (\prod_{i>j}^{k-1} m_{\omega(x,i)}) d^j r_{\omega(x,j)} \Bigr)$
            for all $k \geq 0$.
    \item\label{eq:congruence} For all $x, y \in \Lambda$ the congruence $x \equiv y \mod d^k \Lambda$ holds if and only if $\omega(x,i) = \omega(y,i)$ for all $i=0,\dots,k-1$.
    \item The map $\Omega_k \colon \Lambda/d^k\Lambda \to (\Lambda/d\Lambda)^k$ with $x+d^k\Lambda \mapsto (\omega(x,0),\dots \omega(x,k-1))$ is a well-defined bijection.
   \end{enumerate}
 \end{lemma}
 \begin{proof}
  The first two assertions are easily obtained by induction on $k$. Note that \eqref{eq:congruence} implies that $\Omega_k$ is a well-defined injective map.
  Domain and range of $\Omega_k$ are finite sets with same cardinality, hence $\Omega_k$ is a bijection.
 \end{proof}

 \begin{proof}[Proof of Theorem \ref{thm:stoppingTime}]
  Consider the set $\stp_k(T) = \{\: x \in \Lambda \:|\: \Vert T^k(x) \Vert < \Vert x \Vert\:\}$. Every element in $\stp_k(T)$ has finite stopping time.
  Moreover, we consider the set $A_k \subseteq (\Lambda/d\Lambda)^k$ of sequences $(\omega_0,\dots,\omega_{k-1})$ such that
  \begin{equation*}
     \prod_{i=0}^{k-1} m_{\omega_i} < d^k.
  \end{equation*}
  Suppose $x \in \Lambda$ with $\Omega_k(x+d^k\Lambda) \in A_k$. We show that almost all elements in $y \in x+d^k\Lambda$ satisfy $y \in \stp_k(T)$.
  Indeed, suppose $\Vert y \Vert > \frac{R (M^k - d^k)}{\lambda_k(M-d)}$ where $R = \max_\omega \Vert r_\omega \Vert$, $M = \max_\omega m_\omega$ and
  $\lambda_k = d^k-  \prod_{i=0}^{k-1} m_{\omega(x,i)}$, then by Lemma \ref{lem:stopping} we obtain
  \begin{align*}
     \Vert T^k(y) \Vert &\leq d^{-k} \Vert y \Vert \prod_{i=0}^{k-1} m_{\omega(x,i)}  + d^{-k} \sum_{j=0}^{k-1} M^{k-1-j} d^j R \\
                        &\leq (1 - \lambda_k/d^k) \Vert y \Vert  + \frac{R(M^k-d^k)}{d^k(M-d)}
                        < \Vert y \Vert.
  \end{align*}
  In particular, we may deduce that $\underline{\delta}(\stp_k(T)) > \frac{|A_k|}{d^{ek}}$.
  
  We define $\alpha = \frac{1}{d^e} \sum_{\omega \in \Lambda}\log m_\omega$ and  $\gamma = \log(d) - \alpha$. By assumption $\gamma > 0$.
  Observe that every sequence $(\omega_0,\dots,\omega_{k-1})$ which satisfies
  $| \alpha - \frac{1}{k}\sum_{i = 0}^{k-1} \log(m_{\omega_i}) | < \gamma$ is an element of $A_k$.
  We may consider $X_i = \log m_{\omega_i}$ as independent random variables on the finite uniform probability space $\Lambda/d\Lambda$.
  The weak law of large numbers (see, for instance, Theorem 11.2 in \cite{McCordMoroney}) yields $\lim_{k\to \infty}  \frac{|A_k|}{d^{ek}} = 1$ and
  henceforth $\delta( \stp(T) ) = 1$.
 \end{proof}

\section{Examples with many divergent points}\label{sec:examples}

We consider the Euclidean plane $E = \bbR^2$ with the standard Euclidean norm. Let $\bbZ^2$ be the lattice 
of points with integral coordinates.
Let $d \geq 2$ and $b \geq 1$ be integers. For every pair of integers $i,j \in \{0,\dots,d-1\}$ we define
\begin{align*}
   m_{i,j} = \begin{cases}
                 1      \quad &\text{ if } i=j=0 \\
                 d-1 	\quad &\text{ if } 0< \max(i,j) < d-1\\
                 d+1    \quad &\text{ if } \max(i,j) = d-1.
             \end{cases}
\end{align*}
and shift vectors
\begin{align*}
   r_{i,j} = \begin{cases}
                 0                  \quad &\text{ if } i=j=0 \\
                 (i,  j + bdi)      \quad &\text{ if } 0< \max(i,j) < d-1\\
                 (1, (b+1)d-j)      \quad &\text{ if } i= d-1 \\
                 (d-i, bd(d-i)+1) \quad &\text{ if } j= d-1.\\
             \end{cases}
\end{align*}
Observe that $m_{i,j} (i,j) + r_{i,j} \equiv 0 \bmod d$.
We define the associated generalized Collatz mapping $T\colon \bbZ^2 \to \bbZ^2$ by
\begin{equation}\label{eq:extremalCollatz}
    T(x) = \frac{1}{d}(m_{i,j}x + r_{i,j})  \quad \text{ if } x \equiv (i,j) \bmod d.
\end{equation}

\begin{proposition}\label{prop:extremeExample}
 Let $T$ be the generalized Collatz mapping defined in~\eqref{eq:extremalCollatz}.
 Then
 \begin{equation*}
      \underline{\delta}(\Diver(T))  \geq 1 - \frac{\arccos(\sqrt{1+ (bd)^{-2}}^{-1})}{\pi}.
 \end{equation*}
 In particular, $\underline{\delta}(\Diver(T))$ approaches $1$ as $bd$ tends to infinity.
 If moreover $d \geq 3$, then $\delta(\stp(T)) = 1$.
\end{proposition}
\begin{proof}
 Clearly, the mapping $T$ is of relatively prime type and the set of shift vectors $S_r$ is acute ($S_r$ lies in the first quadrant).
 
 We verify that the product $\prod_{i,j} m_{i,j} = (d+1)^{2d-1} (d-1)^{d^2-2d}$ is less than $d^{d^2}$ for all $d \geq 3$. Indeed, for $d = 3$ it is easy 
 to verify the inequality $2^{13}<3^9$. For $d \geq 4$ we have
 $d^2 - 4d +1 > 0$ and hence
 \begin{equation*}
     (d+1)^{2d-1} (d-1)^{d^2-2d} = (d^2-1)^{2d-1} (d-1)^{d^2-4d+1} < d^{4d -2} d^{d^2-4d + 1} < d^{d^2}. 
 \end{equation*}
 Using Theorem \ref{thm:stoppingTime} we conclude that $\delta(\stp(T)) = 1$.
 
 \medskip

 \noindent\textbf{Claim:} The lines $L_1 = \bbR (1, bd)$ and $L_2 = \bbR (0,1)$ are separating.
 
 \medskip
 
 \noindent Consider the linear forms $\Phi_1(x,y) = y- bd x$ and $\Phi_2(x,y) = x$ with $\ker(\Phi_\ell) = L_\ell$.
 We have $|m_{i,j} \Phi_\ell(x,y)| > |\Phi_{\ell}(r_{i,j})|$ for all $(x,y) \in \bbZ^2\setminus \ker{\Phi_\ell}$. 
 Indeed, this follows from the simple observation that each shift vector $r_{i,j}$ is of the form $(s,t + bds)$ for certain integers $0 \leq s,t < m_{i,j}$.
 
 Let $B^+$ denote the convex cone generated by the shift vectors $S_r$. 
 It is easy to see that $B^+$ is the intersection of the positive half-planes defined by $\Phi_1$ and $\Phi_2$.
 We conclude with Lemma \ref{lem:directed} that all integral points in $D = E \setminus (B^+ \cup -B^+)$ are directed.
 Moreover, the claim implies that $D \cap \bbZ^2$ is $T$-invariant and consequently (by Proposition \ref{prop:InvariantDirected}) consists
 of divergent points.
 
 Let $\beta(1)$ denote the unit disc. In summary we find
 \begin{equation*}
    \underline{\delta}(\Diver(T)) \geq 1 - 2 \frac{\vol(B^+\cap\beta(1))}{\vol(\beta(1))}.
 \end{equation*}
 Finally, we note that $\vol(\beta(1)) = \pi$ and $\vol(B^+ \cap \beta(1)) = \frac{\gamma}{2}$ where $\gamma$ denotes the angle
 between $(1,bd)$ and $(0,1)$. The proposition follows from the fact that $\cos(\gamma) = \frac{bd}{\sqrt{1+b^2d^2}}$.
\end{proof}

\begin{remark}
 It is easy to construct generalized Collatz mappings $T$ with the
 properties obtained in Proposition \ref{prop:extremeExample} on the lattices $\bbZ^e \subseteq \bbR^e$
 for all $e \geq 2$.
 On the other hand, it seems that for $d = 2$ our methods are
 not able to prove the existence of a generalized Collatz mapping $T$ such that
 $\delta(\stp(T)) = 1$ and $\underline{\delta}(\Diver(T)) \geq 1-\varepsilon$. 
 The reason is that for $d=2$ one needs to choose $m_\omega = 1$
 for at least one $\omega \neq 0$ in order to apply Theorem \ref{thm:stoppingTime}.
 However, this choice implies strong restrictions on the separating hyperplanes such that the tame cone becomes small.
 It is possible though to construct generalized Collatz mappings for $d = 2$ with $\delta(\stp(T)) = 1$ and
 $\underline{\delta}(\Diver(T))$ approaching $\frac{1}{2}$.
\end{remark}

\subsection*{Outlook and conclusion}
We have seen that the conjecture stated in the introduction does not generalize to higher dimensional Collatz mappings. 
However, it seems possible that the conjecture holds inside
the wild cones $C_j$ of the decomposition using separating hyperplanes.
Small computer experiments provide some evidence. 

The geometric approach taken in this article might also yield results on more general higher dimensional Collatz mappings.
A possible direction is to replace the numbers $m_\omega$ by integral matrices $M_\omega$ with specific properties.
For instance, the matrices $M_\omega$ could be chosen to be (1) diagonal with integral entries or (2) of the form $m_\omega g_\omega$ where
$m_\omega \geq 1$ is an integer and $g_\omega \in \GL(\Lambda)$. 
In both cases the underlying geometry gets more complicated, so that further restrictions seem necessary.
However, we think that already special cases might yield new interesting phenomena.

\bibliography{library}{}
\bibliographystyle{amsplain}
\end{document}